\documentclass{amsart}
\usepackage{amssymb}
\usepackage{bm}
\usepackage{enumitem}
\usepackage{amsmath}
\usepackage{epsfig}
\usepackage{tikz}
\usepackage[all]{xy}
\usepackage{hyperref}

\newtheorem{theorem}{Theorem}[section]
\newtheorem{prop}[theorem]{Proposition}
\newtheorem{lemma}[theorem]{Lemma}
\newtheorem{cor}[theorem]{Corollary}
\newtheorem*{thm}{Theorem}
\theoremstyle{definition}
\newtheorem{rem}[theorem]{Remark}
\newtheorem{defi}[theorem]{Definition}

\newcommand{\ra}{\rightarrow}

\newcommand{\IP}{ \mathbb{P}}
\newcommand{\IC }{ \mathbb{C}}
\newcommand{\IR}{ \mathbb{R}}

\newcommand{\IZ}{\mathbb{Z}}
\newcommand{\IQ}{\mathbb{Q}}
\newcommand{\IN}{\mathbb{N}}

\newcommand{\coloneqq}{:=}

\DeclareFontFamily{OT1}{pzc}{}
\DeclareFontShape{OT1}{pzc}{m}{it}{<-> s * [1.10] pzcmi7t}{}
\DeclareMathAlphabet{\mathpzc}{OT1}{pzc}{m}{it}
\DeclareMathOperator{\Pic}{Pic}
\DeclareMathOperator{\ns}{NS}
\DeclareMathOperator{\transc}{Tr}

\DeclareMathOperator{\id}{id}

\DeclareMathOperator{\Hom}{Hom}
\DeclareMathOperator{\divi}{div}

\DeclareMathOperator{\Mon}{Mon}
\DeclareMathOperator{\mon}{Mon}
\DeclareMathOperator{\aut}{Aut}

\DeclareMathOperator{\NS}{NS}

\newcommand{\rk}{\mathrm{rk}\,}

\newcommand{\hsk}{K3^{\left[2\right]}}
\newcommand{\hskn}{K3^{\left[n\right]}}

\makeatletter
\def\blfootnote{\xdef\@thefnmark{}\@footnotetext}

\title[Non-symplectic involutions on manifolds of $\hskn$-type]{Non-symplectic involutions on manifolds of $\hskn$-type}
\author{Chiara Camere, Alberto Cattaneo and Andrea Cattaneo} 

\address{Chiara Camere, Max Planck Institute for Mathematics,
Vivatsgasse 7, 53111 Bonn, Germany;
 Dipartimento di Matematica,
Universit\`a degli Studi di Genova, Via Dodecaneso 35, 16146 Genova (GE), Italy } 
\email{camere@dima.unige.it}

\address{Alberto Cattaneo, Institut f\"ur Algebraische Geometrie, Leibniz Universit\"at Hannover, Welfengarten 1, 30167 Hannover, Germany; Dipartimento di Matematica ``F. Enriques'', Universit\`a degli Studi di Milano, Via Cesare Saldini 50, 20133 Milano, Italy; Laboratoire de Math\'ematiques et Applications, Universit\'e de Poitiers, T\'el\'eport 2, Boulevard Marie et Pierre Curie, 86962 Futuroscope Chasseneuil, France.}
\email{cattaneo@math.uni-hannover.de}

\address{Andrea Cattaneo, Institut Camille Jordan,
Universit\'e Claude Bernard Lyon 1,
43 boulevard du 11 novembre 1918,
69100 Villeurbanne, France}
\email{cattaneo@math.univ-lyon1.fr}

\begin{document}


\begin{abstract}
We study irreducible holomorphic symplectic manifolds deformation equivalent to Hilbert schemes of points on a $K3$ surface and admitting a non-symplectic involution. We classify the possible discriminant forms of the invariant and anti-invariant lattice for the action of the involution on cohomology, and explicitly describe the lattices in the cases where the invariant has small rank. We also give a modular description of all $d$-dimensional families of manifolds of $K3^{[n]}$-type with a non-symplectic involution for $d\geq 19$ and $n\leq 5$, and provide examples arising as moduli spaces of twisted sheaves on a $K3$ surface.
\end{abstract}

\maketitle

\section*{Introduction}

The aim of this note is to explain the classification of non-symplectic involutions on IHS manifolds of $\hskn$-type, thus generalizing to all even dimensions the classification which is already known for $n=1$ by foundational work of Nikulin \cite{NikulinInv} on K3 surfaces and for $n=2$ by the work of Beauville \cite{BeauInv} and of Boissi\`ere, the first author and Sarti \cite{BCS}. The core of the classification result contained in this work comes from Joumaah's PhD thesis \cite{joumaah}, but he kindly decided to let us publish by ourselves. On the other hand, the proof of one of the main results in loc.\ cit.\ is not entirely correct, so in this paper we prove a revised statement (Proposition \ref{discr groups involutions}), in order to obtain the correct classification of non-symplectic involutions on manifolds of $\hskn$-type. 

In the first two authors' work \cite{CC} the interested reader can find the analogue classification for non-symplectic automorphisms of odd prime order: although the lattice-theoretical techniques used are similar, and descend from  work by Nikulin \cite{nikulin}, the prime $p=2$ is somewhat different with respect to other primes because for $n\geq 2$ it always divides $2(n-1)$, which is the discriminant of the Beauville--Bogomolov--Fujiki lattice $L_n:=U^{\oplus 3}\oplus E_8^{\oplus 2}\oplus \langle -2(n-1)\rangle$, i.e.\ the second cohomology lattice of any manifold of $\hskn$-type.

Concerning involutions, in \cite{catt_autom_hilb} the second author computed the automorphism group of the Hilbert scheme of $n$ points over a generic projective $K3$ surface, showing that this group (if not trivial) is generated by exactly one non-natural and non-symplectic involution (for $n=2$, this had already been proved by Boissi\`ere, the third author, Nieper-Wisskirchen and Sarti \cite{bcnws}). The present paper also provides a partial extension of these results, allowing the pair consisting of a Hilbert scheme and its involution to be deformed.

\subsection*{IHS manifolds and automorphisms}

We recall that an irreducible holomorphic symplectic (IHS) manifold is a compact complex K\"ahler manifold $X$ which is simply connected and such that $H^{2, 0}(X)$ is generated by the class of a single holomorphic symplectic (i.e.\ everywhere non-degenerate) $2$-form. Basic examples of IHS manifolds are provided by $K3$ surfaces and, in dimension $2n$, by the Hilbert scheme of zero-dimensional subschemes of length $n$ of a $K3$ surface. As small deformations of IHS manifolds are still IHS, we can then produce new examples: we say that an IHS manifold is of $\hskn$-type if it is deformation equivalent to the Hilbert scheme of $n$ points on a $K3$ surface.

The deformation theory of IHS manifolds is sufficiently well understood. For any manifold $X$ of $\hskn$-type, a \emph{marking} is a lattice isometry $\eta: H^2(X, \IZ) \longrightarrow L_n$, where we recall that $H^2(X, \IZ)$ is a lattice by means of the Beauville--Bogomolov--Fujiki form (see \cite[$\S$8]{Beauville}). Then, there exists a well-defined compact complex moduli space which parametrizes marked IHS manifolds of $\hskn$-type. A fundamental result, due to work by Huybrechts, Markman and Verbitsky, is the Global Torelli Theorem \cite[Corollary 1.20]{VerbitskyTorelli}, which describes the fibers of the period map associated to this moduli space. 

The use of markings allows us to transfer most of the questions about automorphisms to a purely algebraic setting, involving lattices and their properties. However, we need to determine which of the isometries of the abstract lattice $L_n$ correspond, via the marking, to automorphisms of the IHS manifold. To this end, we will make use of Markman's version of the Torelli Theorem \cite[Theorem 1.3]{markman}.

\subsection*{Structure of the paper and main results}

Our study of involutions on manifolds of $\hskn$-type will be conducted in two steps. In Section \ref{sect: lattices} we study the problem only from a lattice-theoretical point of view: our aim is to classify the possible discriminant groups of pairs $T, S \subset L_n$ consisting of the invariant lattice $T$ and the anti-invariant (or co-invariant) lattice $S$ of a non-symplectic involution. We provide this classification in Proposition \ref{discr groups involutions}, fixing the inaccuracies of \cite{joumaah}. An important ingredient of our proof is the fact that one between the invariant and anti-invariant lattice is $2$-elementary (Proposition \ref{prop: T o S 2-elementary}).

In Section \ref{sec: existence}, by using the Global Torelli Theorem we prove that the conditions determined in Section \ref{sect: lattices} on the abstract lattices $T,S$ are also sufficient to obtain a marked manifold of  $\hskn$-type with a non-symplectic involution, having $T$ and $S$ as invariant and co-invariant lattice respectively.

\begin{thm}[{Theorem \ref{thm: existence}}]
Let $\rho \in O(L_n)$ be an involution whose invariant lattice $T$ is hyperbolic with $\rk(T) \leq 20$. Assume also that $\rho\vert_{A_L} = \pm \id$. Then there exists a marked manifold $(X, \eta)$ of $\hskn$-type with a non-symplectic involution $i \in \aut(X)$ such that $\eta \circ i^* = \rho \circ \eta$.
\end{thm}

In Section \ref{sect: geography} we focus on the cases where the invariant lattice has small rank, i.e.~$\rk(T) = 1$ or $2$. For $2 \leq n \leq 5$ we explicitly classify the isometry classes of the pairs of lattices $T,S$ (Propositions \ref{prop: rank one}, \ref{prop:rank two +id} and \ref{prop:rk two -id}). Non-symplectic involutions of manifolds of $\hskn$-type having invariant lattice of small rank are particularly interesting, since they deform in families of large dimensions. For each  possible action on cohomology $\rho \in O(L_n)$ in our classification, we study the corresponding \mbox{moduli} space $\mathcal{M}_{T, \rho}$ of $(\rho, T)$-polarized manifolds of $\hskn$-type with a non-symplectic involution. 

\begin{thm}[{Theorem \ref{thm: max dim families}}]
Let $(X, \eta)$ be a marked manifold of $K3^{[n]}$-type for \mbox{$2 \leq n \leq 5$}, and let $i \in \aut(X)$ be a non-symplectic involution such that the pair $(X, i)$ deforms in a family of dimension $d \geq 19$. Then $(X, \eta)$ belongs to the closure of one of the following moduli spaces:
 \begin{enumerate}
  \item[$n=2$:] $\mathcal{M}_{\langle 2\rangle, \rho_a}$ or $\mathcal{M}_{U(2), \rho_1}$
  \item[$n=3$:] $\mathcal{M}_{\langle 2\rangle, \rho_a}$, $\mathcal{M}_{\langle 4\rangle, \rho}$ or $\mathcal{M}_{U(2), \rho_1}$
  \item[$n=4$:] $\mathcal{M}_{\langle 2\rangle,\rho_a}$, $\mathcal{M}_{\langle 2\rangle,\rho_b}$  or $\mathcal{M}_{U(2), \rho_1}$
  \item[$n=5$:] $\mathcal{M}_{\langle 2\rangle, \rho_a}$, $\mathcal{M}_{U(2), \rho_1}$ or $\mathcal{M}_{U(2), \rho_2}$
 \end{enumerate}  
 where $\rho,\rho_a,\rho_1,\rho_2$ are defined in Remark \ref{rem: moduli spaces recap}.
 
All these moduli spaces are irreducible with the exception of $\mathcal{M}_{U(2), \rho_2}$ for $n=5$, which has three distinct irreducible components.
\end{thm}

Finally, in Section \ref{sect: examples} we use moduli spaces of twisted sheaves on $K3$ surfaces to describe the generic element in the maximal moduli spaces $\mathcal{M}_{T, \rho}$ of dimension $19$ (Propositions \ref{prop: twisted induced U(2) irred} and \ref{prop: twisted induced U(2) n=5}), though only in one case the involution is induced by a non-symplectic involution of the underlying $K3$ surface. Finding an explicit description of the automorphism in the other families is still an open problem.

\subsection*{Notations and conventions}

Throughout the paper, all the varieties will be defined over the field $\IC$ of complex numbers.

A \emph{lattice} is a free abelian group $M$ equipped with a symmetric non-degenerate bilinear form $\left(\cdot, \cdot\right): M \times M \ra \IZ$. Its discriminant group $A_M$ is defined as $A_M = M^\vee / M$, where $M^\vee = \Hom_\IZ(M, \IZ)$ is the dual group of $M$. If $A_M$ is cyclic of order $m$, we write $A_M \cong \frac{\IZ}{m\IZ}\left(\alpha\right)$ if the finite quadratic form $q_M: A_M \ra \IQ / 2\IZ$ (induced by the quadratic form on $M$) takes value $\alpha$ on a generator of $A_M$. For any positive integer $n \geq 2$, we will denote by $L_n$ the lattice
\[L_n = U^{\oplus 3} \oplus E_8^{\oplus 2} \oplus \langle -2(n - 1) \rangle,\]
where $U$ is the hyperbolic plane, $E_8$ is the unique unimodular lattice of signature $(0, 8)$ and for any integer $t \neq 0$ we denote by $\langle t \rangle$ the lattice generated by an element $\delta$ with $(\delta, \delta) = t$.

For a pair of lattices $M$, $N$ there may be several non-isometric embeddings of $M$ into $N$. When we say that $M$ is embedded in $N$, writing $M \subset N$, we always mean that an embedding $j: M \hookrightarrow N$ has been fixed. We will consider two such embeddings $j, j'$ as being isomorphic if there exist isometries $\psi \in O(M)$ and $\varphi \in O(N)$ such that $j \circ \psi = \varphi \circ j'$. The images $j(M)$, $j'(M)$ inside $N$ are also called \emph{isomorphic sublattices} according to \cite[$\S$1.5]{nikulin}.

\subsection*{Acknowledgements}

Chiara Camere is grateful to Max Planck Institute for Mathematics in Bonn for its hospitality and financial support. Alberto Cattaneo was supported by Universit\`a degli Studi di Roma Tor Vergata as a member of the project ``Families of curves: their moduli and their related varieties'' (Codice Unico Progetto:  E81$\vert$18000100005, in the framework of \emph{Mission Sustainability 2017} -- Tor Vergata University of Rome, PI: Prof.\ Flaminio Flamini). Andrea Cattaneo is supported by the LABEX MILYON (ANR-10-LABX-0070) of Universit\'e de Lyon, within the program ``Investissements d'Avenir'' (ANR-11-IDEX- 0007) operated by the French National Research Agency (ANR) and is member of GNSAGA of INdAM. The authors are grateful to Alessandra Sarti and Bert van Geemen for reading the paper and for their suggestions, as well as to Giovanni Mongardi for useful remarks.

\section{Involutions of the lattice \texorpdfstring{$L_n$}{Ln}}\label{sect: lattices}

\subsection{Invariant and anti-invariant lattices} \label{sec: inv and anti-inv lattices}

Let $(X, i)$ be a pair consisting of an IHS manifold $X$ of $\hskn$-type and a non-symplectic involution $i\in \aut(X)$. The lattice $H^2(X, \IZ)$ is isometric to $L_n = U^{\oplus 3} \oplus E_8^{\oplus 2} \oplus \langle -2(n - 1) \rangle$, as we already recalled, and $i^* \in \Mon^2(X)$, which is the subgroup of monodromy operators inside $O(H^2(X,\IZ))$. We now fix $n\geq 2$ and we write $L:=L_n$ for the sake of simplicity. 

By \cite[Cor.\ 9.5(1)]{markman} we have a primitive embedding $L \hookrightarrow M$ where $M \coloneqq U^{\oplus 4} \oplus E_8^{\oplus 2}$ is the Mukai lattice, unimodular of rank $24$. Observe that, if we call $\delta$ a generator of $\langle -2(n - 1) \rangle$ in $L$, then $A_L$ is cyclic generated by  $\frac{1}{2(n - 1)}\delta $, i.e.\
\begin{equation}\label{eq: discriminant L}
A_L = \frac{\IZ}{2(n - 1) \IZ}\left( -\frac{1}{2(n - 1)} \right).
\end{equation}
We denote by $L^\perp$ the orthogonal complement of $L$ inside $M$. By \cite[Cor.\ 1.6.2]{nikulin} we have
\begin{equation}\label{eq: discriminant lperp}
A_{L^{\perp}} \cong \frac{\IZ}{2(n - 1)\IZ} \left( \frac{1}{2(n - 1)} \right).
\end{equation}
Since $L^{\perp} \subset M$ has rank one, we deduce that $L^{\perp} \cong \langle 2(n - 1) \rangle$.

After choosing a marking (i.e.\ an isometry) $\eta: H^2(X, \IZ) \ra L$, we can consider the action $i^* \in O(L)$. By \cite[Lemma 9.2]{markman}, $i^*$ satisfies the following properties: it has spin norm equal to $1$ (equivalently, it is orientation preserving) and it induces $\pm \id$ on the discriminant group $A_L$. This means that $\pm i^* \in \widetilde{O}(L)$, where for any lattice $\Lambda$ the stable orthogonal group $\widetilde{O}(\Lambda)$ is the subgroup of $O(\Lambda)$ consisting of isometries that induce the identity on the discriminant group $A_\Lambda$. Let $\sigma = \pm i^*$ be such that $\sigma \in \widetilde{O}(L)$.

The \emph{invariant lattice} of the involution $i \in \aut(X)$ is the sublattice $H^2(X, \IZ)^{i^*} \subset H^2(X, \IZ)$ of elements that are fixed by $i^*$. Its orthogonal complement in $H^2(X, \IZ)$ is called the \emph{anti-invariant} (or co-invariant) lattice. Notice that the anti-invariant lattice coincides with $\ker(\id + i^*)$ (see \cite[\S 5]{smith}) and therefore it is equal to $H^2(X, \IZ)^{-i^*}$, the invariant lattice of $-i^*$.

We now show that one between the invariant and the anti-invariant lattice of $i^*$ is $2$-elementary.

\begin{prop} \label{prop: T o S 2-elementary}
Let $X$ be a manifold of $K3^{[n]}$-type and let $i \in \aut(X)$ be a non-symplectic involution. Then one of the following holds:
\begin{enumerate}
\item $i^*$ acts as $\id$ on the discriminant group of $H^2(X, \IZ)$ and $H^2(X, \IZ)^{-i^*}$ is $2$-e\-le\-men\-ta\-ry;
\item $i^*$ acts as $-\id$ on the discriminant group of $H^2(X, \IZ)$ and $H^2(X, \IZ)^{i^*}$ is $2$-elementary.
\end{enumerate}
\end{prop}
\begin{proof}
Consider $\sigma \in \widetilde{O}(L)$ as above: in both cases we want to show that the invariant lattice of $-\sigma$ is $2$-elementary. By \cite[Lemma 7.1]{ghs2}, we can extend $\sigma$ to an isometry $\tau \in \widetilde{O}(M)$ such that $\tau\vert_{L^{\perp}} = \id_{L^{\perp}}$ and with the following properties:
\begin{enumerate}
\item $L^\sigma \subset M^\tau$;
\item $L^{-\sigma} \subset M^{-\tau}$;
\item $L^{\perp} \subset M^\tau$.
\end{enumerate}

As a consequence, $L^\sigma \oplus L^{\perp} \subset M^\tau$ is a finite index sublattice and moreover, inside the lattice $M$:
\[M^{-\tau} = (M^\tau)^{\perp} \subset (L^\sigma \oplus L^{\perp})^{\perp} = (L^ \sigma)^{\perp} \cap L \subset L.\]
Hence $L^{-\sigma} = M^{-\tau}$. The invariant and anti-invariant lattices of an involution of an even unimodular lattice are $2$-elementary by \cite[Lemma 3.5]{ghs3}: this concludes the proof.
\end{proof}

With the same notation used above, we remark the following facts.

\begin{lemma}\label{lemma: orthogonal Lperp}
\begin{enumerate}\leavevmode
 \item The lattice $L^\sigma$ is primitively embedded in $M^\tau$.
 \item The lattice $L^{\perp}$ is primitively embedded in $M^\tau$.
 \item The lattices $L^\sigma$ and $L^{\perp}$ are the orthogonal complement of each other in $M^\tau$.
\end{enumerate}
\end{lemma}

\begin{proof}\leavevmode
\begin{enumerate}
\item As $L^\sigma \subset L$ and $L \subset M$ are primitive, we deduce that $L^\sigma \subset M$ is primitive. The claim follows then from the inclusion $L^\sigma \subset M^\tau$.
\item This follows from the fact that $L^{\perp} \subset M$ is primitive and $L^{\perp} \subset M^\tau \subset M$.
\item Since $(L^\sigma, L^{\perp}) = 0$, we deduce that $L^{\perp} \subset (L^\sigma)^{\perp_{M^\tau}}$. Moreover, both $L^{\perp}$ and $(L^\sigma)^{\perp_{M^\tau}}$ are primitive sublattices of $M^\tau$: since they have the same rank, they must coincide.\qedhere
\end{enumerate}
\end{proof}

In the same spirit of \cite[Def.~4.1]{BCS}, we give the following definition.

\begin{defi}
An automorphism $f$ of a manifold $X$ of  $\hskn$-type is \emph{natural} if there exists a $K3$ surface $\Sigma$ and $\varphi\in\aut(\Sigma)$ such that $(X, f)$ is deformation equivalent to $(\Sigma^{[n]}, \varphi^{[n]})$.
\end{defi}

\begin{lemma}
Let $X$ be a manifold of  $\hskn$-type and $i \in \aut(X)$ be a natural non-symplectic involution. Then $i^* \in \widetilde{O}(H^2(X, \IZ))$.
\end{lemma}
\begin{proof}
As shown in \cite[\S 4]{BCS}, the isomorphism class of the invariant lattice of a non-symplectic involution is deformation invariant. For the pair $(\Sigma^{[n]}, \varphi^{[n]})$, the action of the natural involution on the exceptional divisor of the Hilbert--Chow morphism $\Sigma^{[n]} \ra \Sigma^{(n)}$ is trivial by \cite[Thm.~1]{bs}. Let $\delta \in H^2(\Sigma^{[n]}, \IZ)$ be the class whose double is the exceptional divisor. From $i^*(2 \delta) = 2 \delta$ we get that the image of $L + \frac{1}{2(n - 1)}\delta \in A_L$ is $L + \frac{1}{2(n - 1)}\delta$, hence the action of $i^*$ on $A_L$ is trivial.
\end{proof}

\begin{cor}
Let $X$ be a manifold of  $\hskn$-type and $i \in \aut(X)$ be a natural non-symplectic involution. Then the anti-invariant lattice $H^2(X, \IZ)^{-i^*}$ is $2$-elementary.
\end{cor}

\subsection{Discriminant groups}\label{sec: discr groups}
We explain in this section the inaccuracies in the proof of \cite[Prop.~5.1.1]{joumaah} and provide the necessary corrections. Adopting our notation, which differs from the one used by Joumaah, let $X$ be a manifold of $\hskn$-type with a non-symplectic involution $i \in \aut(X)$. Let $T = L^{i^*}$, $S = L^{-i^*}$ be, respectively, the invariant and anti-invariant lattices of the involution. The aim of \cite[Prop.~5.1.1]{joumaah} is to classify the discriminant groups $A_T, A_S$. In order to do so, Joumaah considers the isotropic subgroup $H_L \subset A_T \oplus A_S$, which is isomorphic to $\frac{L}{T \oplus S} \cong \left( \frac{\IZ}{2 \IZ} \right)^a$  for some $a \geq 0$, and its projections \mbox{$H_T \coloneqq p_T(H_L) \subset A_T$}, \mbox{$H_S \coloneqq p_S(H_L) \subset A_S$}. In particular, $H_L \cong H_T \cong H_S$ as groups, $\frac{H_L^\perp}{H_L} \cong A_L$ and $\gamma \coloneqq p_S \circ p_T^{-1} : H_T \ra H_S$ is an anti-isometry.

The following proposition provides the complete classification for the discriminant groups $A_T$, $A_S$. We refer to \cite[Prop.~3.2]{CC} for the analogous classification in the case of automorphisms of odd prime order.

\begin{prop}\label{discr groups involutions}
Let $X$ be a manifold of $\hskn$-type, for $n \geq 2$, and let $l \geq 1$ and $m$ odd such that $2(n-1)=2^l m$. Let $\mathcal{G} \subset \aut(X)$ be a group of order $2$ acting non-symplectically on $X$. Denote by $T,S \subset L \coloneqq L_n$, respectively, the invariant and anti-invariant sublattices for the action of $\mathcal{G}$, with $\frac{L}{T \oplus S} \cong \left( \frac{\IZ}{2 \IZ} \right)^a$ for some $a \geq 0$. Then one of the following cases holds:
\begin{enumerate}
\item[\textit{(i)}] $A_T \cong \left(\frac{\IZ}{2 \IZ}\right)^{\oplus a} \oplus \frac{\IZ}{2(n-1) \IZ}$, $A_S \cong \left(\frac{\IZ}{2 \IZ}\right)^{\oplus a}$ or vice versa;
\item[\textit{(ii)}] $a \geq 1$, $A_T \cong \left(\frac{\IZ}{2 \IZ}\right)^{\oplus a-1} \oplus \frac{\IZ}{2(n-1) \IZ}$, $A_S \cong \left(\frac{\IZ}{2 \IZ}\right)^{\oplus a+1}$ or vice versa;
\item[\textit{(iii)}] $l = 1$, $a=0$, $A_T \cong \frac{\IZ}{m \IZ}$, $A_S \cong \frac{\IZ}{2 \IZ}$ or vice versa.
\end{enumerate}

\begin{proof}
Let $i$ be the non-symplectic involution generating the group $\mathcal{G}$ and, as before, let $\sigma = \pm i^*$ be the isometry such that $\sigma \in \widetilde{O}(L)$. Let $T, S$ be the invariant and anti-invariant lattices of $i^*$, as in the statement. If $\sigma = i^*$, then $T = L^\sigma$, $S = L^{-\sigma}$; if $\sigma = -i^*$, then $T = L^{-\sigma}$, $S = L^\sigma$.

As we showed in Proposition \ref{prop: T o S 2-elementary}, the lattice $L^{-\sigma}$ is $2$-elementary, therefore $A_{L^{-\sigma}}$ coincides with its Sylow $2$-subgroup (it actually coincides with its $2$-torsion part). Moreover, $A_{L^\sigma} = \left( A_{L^\sigma} \right)_2 \oplus \frac{\IZ}{m \IZ}$, where $\left( A_{L^\sigma} \right)_2$ denotes the Sylow $2$-subgroup of $A_{L^\sigma}$ (see \cite[Prop.~5.1.1]{joumaah}). Using the notation introduced at the beginning of the section, there exist subgroups $H_{L^{\sigma}} \subset \left( A_{L^\sigma} \right)_2$ and $H_{L^{-\sigma}} \subset A_{L^{-\sigma}}$ isomorphic to $\frac{L}{L^{\sigma} \oplus L^{-\sigma}} \cong \left( \frac{\IZ}{2 \IZ} \right)^{a}$. The case $l = 1$ was correctly discussed by Joumaah in his proof: the only possibilities are $A_{L^{\sigma}} = H_{L^{\sigma}} \oplus \frac{\IZ}{2 \IZ} \oplus \frac{\IZ}{m \IZ}$, $A_{L^{-\sigma}} = H_{L^{-\sigma}}$ (i.e.\ case (i) of the statement), or $A_{L^{\sigma}} = H_{L^{\sigma}} \oplus \frac{\IZ}{m \IZ}$, $A_{L^{-\sigma}} = H_{L^{-\sigma}} \oplus \frac{\IZ}{2 \IZ}$ (case (ii) if $a \geq 1$ or case (iii) if $a = 0$).

If $l \geq 2$, we define $G \coloneqq (A_{L^{\sigma}})_2 \oplus A_{L^{-\sigma}}$ and let $G_2 \subset G$ be the subgroup of elements of order $2$ in $G$. Joumaah showed that $[G : G_2] = 2^{l-1}$, which implies $G \cong \left( \frac{\IZ}{2 \IZ} \right)^{2a} \oplus \frac{\IZ}{2^l \IZ}$. As a consequence, we obtain two possible structures (not just one, as stated in \cite[Prop.~5.1.1]{joumaah}; see below for details) for the summands $(A_{L^{\sigma}})_2$ and $A_{L^{-\sigma}}$ of $G$, recalling that $L^{-\sigma}$ is $2$-elementary and that both $(A_{L^{\sigma}})_2$, $A_{L^{-\sigma}}$ contain a subgroup isomorphic to $\left( \frac{\IZ}{2 \IZ} \right)^{a}$:
\begin{itemize}
\item $(A_{L^{\sigma}})_2 = \left( \frac{\IZ}{2 \IZ} \right)^{a} \oplus \frac{\IZ}{2^l \IZ}, \; A_{L^{-\sigma}} = \left( \frac{\IZ}{2 \IZ} \right)^{a}$ (case (i));
\item $a \geq 1$, $(A_{L^{\sigma}})_2 = \left( \frac{\IZ}{2 \IZ} \right)^{a-1} \oplus \frac{\IZ}{2^l \IZ}, \; A_{L^{-\sigma}} = \left( \frac{\IZ}{2 \IZ} \right)^{a+1}$ (case (ii)).\qedhere
\end{itemize}
\end{proof}
\end{prop}

\begin{rem}
Assume that $i^* \in \widetilde{O}(L)$, so that $\sigma = i^*$. If $l > 1$, Joumaah correctly highlighted in his proof that the index $[G:G_2]$ needs to be $2^{l-1}$ and therefore $G \cong \left( \frac{\IZ}{2 \IZ} \right)^{2a} \oplus \frac{\IZ}{2^l \IZ}$. However, contrary to what he stated, this does not necessarily imply that $G = H_T \oplus H_S \oplus \frac{\IZ}{2^l \IZ}$, from which he inferred $A_T \cong H_T \oplus A_L, A_S = H_S$ as the only possibility for the discriminant groups. Indeed, we exhibit two lattices $T,S$ which are the invariant and anti-invariant lattices of a non-symplectic involution of a manifold of $K3^{[3]
}$-type and whose discriminant groups are in contrast with \cite[Prop.~5.1.1]{joumaah}.

For $n=3$ we have $2(n-1)=4$, meaning $l=2$, $m=1$. The authors of \cite{IKKR} describe a $20$-dimensional family of manifolds of $K3^{\left[3\right]}$-type, called \emph{double EPW cubes}, with polarization of degree four and divisibility two (see \cite[Prop.~5.3]{IKKR}), whose members are always endowed with a non-symplectic involution $i$. As a consequence, the invariant lattice of $i$ is $T \cong \langle 4 \rangle$ and the anti-invariant lattice is $S \cong U^{\oplus 2} \oplus E_8^{\oplus 2} \oplus \langle -2 \rangle^{\oplus 2}$. In particular, their discriminant groups are:
\[ A_T = \langle t \rangle \cong \frac{\IZ}{4 \IZ}\left( \frac{1}{4} \right), \qquad A_S = \langle s_1, s_2 \rangle \cong \frac{\IZ}{2 \IZ}\left( -\frac{1}{2} \right) \oplus \frac{\IZ}{2 \IZ}\left( -\frac{1}{2} \right).\]
In this case $G = A_T \oplus A_S$, since $m=1$. Moreover
$$16 = \left| A_T \oplus A_S \right| = \left[ L : T \oplus S \right]^2 \left| A_L \right| = 2^{2a} \cdot 4 $$
\noindent therefore $a=1$.  Looking at the discriminant quadratic forms on $A_T$ and $A_S$, the only possible choice for the subgroups of order two $H_T \subset A_T$ and $H_S \subset A_S$, with $H_T \cong H_S(-1)$, is the following:
\[ H_T = \langle 2t \rangle \subset A_T, \quad H_S = \langle s_1 + s_2 \rangle \subset A_S\]
\noindent which implies $H_L = \langle 2t + s_1 + s_2 \rangle \subset A_T \oplus A_S$. One can check, by computing $H_L^\perp \subset A_T \oplus A_S$, that $\frac{H_L^\perp}{H_L} \cong A_L = \frac{\IZ}{4 \IZ}\left( -\frac{1}{4} \right)$.

This is therefore a case where $l = 2 > 1$ and $[G : G_2 ] = 2 = 2^{l-1}$. However, it is not possible to write the group $G = A_T \oplus A_S$ as $G = H_T \oplus H_S \oplus \frac{\IZ}{2^l \IZ}$ and it is not true that $A_T \cong H_T \oplus A_L$, $A_S = H_S$.
\end{rem}

\begin{rem}
In the case of manifolds of $\hsk$-type, it was proved in \cite[Lemma 8.1]{BCS} (extending results from \cite[\S 6]{smith}) that the discriminant groups can only be $A_S \cong \left(\frac{\IZ}{2 \IZ}\right)^{\oplus a}$, $A_T \cong \left(\frac{\IZ}{2 \IZ}\right)^{\oplus a+1}$ or vice versa. This is coherent with the classification of Proposition \ref{discr groups involutions} (if $n=2$ we have $2(n-1) = 2$, hence $l = m = 1$).
\end{rem}

\section{Existence of automorphisms} \label{sec: existence}

In this section we show that the lattice-theoretic conditions of Proposition~\ref{prop: T o S 2-elementary} are actually sufficient to give rise to a geometric realization. First, we prove that every $2$-elementary sublattice of $L = L_n$ is the invariant (or anti-invariant) lattice of some involution of $L$, and finally that we can generically lift this abstract involution to an involution of a manifold of $\hskn$-type.

\begin{prop} \label{prop: extension involution}
Let $S$ be an even $2$-elementary lattice, primitively embedded into an even lattice $\Lambda$. Then $\id_{S^{\perp}} \oplus (-\id_S)$ (resp.\ $(-\id_{S^{\perp}}) \oplus \id_S$) extends to an isometry $\rho \in \widetilde{O}(\Lambda)$ (resp.\ $-\rho \in \widetilde{O}(\Lambda)$).
\end{prop}
\begin{proof}
By \cite[Thm.~1.1.2]{nikulin}, we can primitively embed $\Lambda$ into an even unimodular lattice $V$ of high enough rank. We fix such a primitive embedding and consider the orthogonal complements $\Lambda^{\perp_V}$ and $S^{\perp_V}$ of $\Lambda$ and $S$ inside $V$. Obviously, $V$ is an overlattice of $S \oplus S^{\perp_V}$. We want to show that $\alpha \coloneqq \id_{S^{\perp_V}} \oplus (-\id_S)$ extends to $M$. A completely analogous proof will show that also $(-\id_{S^{\perp_V}}) \oplus \id_S$ extends, as in the statement. Let $H_V = V/(S \oplus S^{\perp_V})$ be the isotropy subgroup of $A_S \oplus A_{S^{\perp_V}}$ corresponding to the overlattice $V$ and let $p_S$, $p_{S^{\perp_V}}$ be the two projections to $A_S$ and $A_{S^{\perp_V}}$:
\[\xymatrix{H_V = V/(S \oplus S^{\perp_V}) \ar[d]^{p_{S^{\perp_V}}} \ar[dr]^{p_S} & \subset A_S \oplus A_{S^{\perp_V}}\\
H_{S^{\perp_V}} \subset A_{S^{\perp_V}} & H_S \subset A_S.}\]
Since $V$ is unimodular, we have $H_{S^{\perp_V}} = A_{S^{\perp_V}}$ and $H_S = A_S$. As before, let $\gamma: A_{S^{\perp_V}} \longrightarrow A_S$ be the anti-isometry  given by $p_S\circ (p_{S^{\perp_V}})^{-1}$. By \cite[Prop.~1.5.1]{nikulin}, the existence of an extension of $\alpha$ to $V$ is equivalent to the commutativity of the diagram
\[\xymatrix{A_{S^{\perp_V}} \ar[r]^\gamma \ar[d]^{\overline{\id}_{S^{\perp_V}}} & A_S \ar[d]^{-\overline{\id}_S}\\
A_{S^{\perp_V}} \ar[r]^\gamma & A_S}\]
\noindent where, for any lattice $N$ and $\mu \in O(N)$, we denote by $\overline{\mu}$ the isometry of finite quadratic forms induced by $\mu$ on the discriminant group $A_N$.
The diagram is commutative because $-\gamma = \gamma$, since $S$ is $2$-elementary, hence we get the extension $\widetilde{\alpha}\in O(V)$ of $\alpha$ to $V$.

As $S^{\perp_\Lambda} \oplus \Lambda^{\perp_V} \subset S^{\perp_V}$, we deduce that $\Lambda^{\perp_V}$ is invariant for the action of $\widetilde{\alpha}$. Let $\rho$ be the restriction $\widetilde{\alpha}\vert_{\Lambda}$. Since $\rho \oplus \id_{\Lambda^{\perp_V}}$ extends to $\widetilde{\alpha} \in O(V)$, we have a commutative diagram
\[\xymatrix{A_\Lambda \ar[r]^\beta \ar[d]_{\overline{\rho}} & A_{\Lambda^{\perp_V}} \ar[d]^{\overline{\id}_{\Lambda^{\perp_V}}}\\
A_\Lambda \ar[r]^\beta & A_{\Lambda^{\perp_V}}}\]
where $\beta:= p_{\Lambda^{\perp_V}}\circ (p_\Lambda)^{-1}$. Hence, $\overline{\rho} = \id_{A_\Lambda}$, i.e.\ $\rho \in \widetilde{O}(\Lambda)$.
\end{proof}

\begin{rem}
This is in some sense a converse of \cite[Lemma~3.5]{ghs3}. See also \cite[Prop.~1.5.1]{dolgachev}.
\end{rem}

We come now to the second part of the section. First, we recall some results on lattice-polarized manifolds of $\hskn$-type.

Let $T$ be a hyperbolic lattice which admits a primitive embedding $j:T \hookrightarrow L$, with $\rk(T) \leq 20$. We identify $T$ with the sublattice $j(T) \subset L$ and we denote by $S$ its orthogonal complement in $L$. Following \cite[\S 4.1]{joumaah}, we say that $T$ is \emph{admissible} if it is the invariant lattice of a monodromy operator $\rho \in \mon^2(L)$ of order two. In particular, $T$ and $S$ are as in Proposition \ref{discr groups involutions}, therefore one of them is $2$-elementary. This implies, by Proposition \ref{prop: extension involution}, that $\rho$ is the unique extension of $\id_T\oplus(-\id_S)$ to $L$.

Let $X$ be a manifold of $K3^{[n]}$-type and $i \in\aut(X)$ be a non-symplectic involution acting on it. Joumaah says that the pair $(X,i)$ is \emph{of type $T$} if it admits a $(\rho,T)$-polarization, i.e.\ a marking $\eta: H^2(X,\mathbb{Z})\rightarrow L$ such that $\eta \circ i^\ast= \rho\circ\eta$. If $(X,i)$ and $(X',i')$ are two pairs of type $T$, they are said to be isomorphic if there exists an isomorphism $f: X\rightarrow X'$ such that $i' = f\circ i \circ f^{-1}$. The monodromy operators $f^* \in \mon^2(L)$ induced by these isomorphisms of pairs are the isometries contained in
\[\mon^2(L,T) \coloneqq \left\{ g\in\mon^2(L) \mid g\circ \rho=\rho\circ g \right\} = \left\{g\in\mon^2(L) \mid g(T) = T \right\}.\]

In particular, for any $g \in \mon^2(L,T)$ we have that $g \vert_T \in O(T)$ and $g \vert_S \in O(S)$. We can then define the following subgroups:
\[ \Gamma_T \coloneqq \left\{ g \vert_T \mid g \in \mon^2(L,T) \right\} \subset O(T),  \qquad \Gamma_S \coloneqq \left\{ g \vert_S \mid g \in \mon^2(L,T) \right\} \subset O(S).\]
Notice that local deformations of a pair $(X,i)$ of type $T$ are parametrized by $H^{1,1}(X)^{i^*}$ (more details on this are provided in \cite[Theorem 2]{BeauInv} and \cite[\S 4]{BCS}).

Inside the moduli space $\mathcal{M}_L$ of marked IHS manifolds of $\hskn$-type, let $\mathcal{M}_{T,\rho}$ be the subspace of $(\rho,T)$-polarized marked manifolds $(X,\eta)\in\mathcal{M}_L$. Since the symplectic form $\omega_X$ generating $H^{2,0}(X)$ is orthogonal to the N\'eron--Severi group (which contains $T$), for any $(X,\eta)\in\mathcal{M}_{T,\rho}$ the period point $\eta(H^{2,0}(X))$ belongs to 
\[ \Omega_S \coloneqq \left\{ \kappa \in \mathbb{P}(S \otimes \IC) \mid (\kappa, \kappa) = 0, (\kappa, \overline{\kappa}) > 0\right\}.\]

Moreover, by \cite[Proposition 4.6.7]{joumaah}, the period map restricts to a holomorphic surjective morphism 
$$\mathcal{P}: \mathcal{M}_{T,\rho}\longrightarrow \Omega_S^0 \coloneqq \Omega_S\setminus \bigcup_{\delta\in\Delta(S)}(\delta^{\perp}\cap \Omega_S),$$ where $\Delta(S)$ is the set of wall divisors (i.e.\ primitive integral monodromy birationally minimal classes) contained in $S$. This restriction is equivariant with respect to the action of $\mon^2(L,T)$, hence we also obtain a surjection
\[
\mathcal{P}: \mathcal{M}_{T,\rho}/\mon^2(L,T)\longrightarrow \Omega_S^0/\Gamma_S.
\]

\begin{theorem}\label{thm: existence}
Let $\rho \in O(L)$ be an involution whose invariant lattice $T$ is hyperbolic with $\rk(T) \leq 20$. Assume also that $\pm \rho \in \widetilde{O}(L)$. Then there exists a marked manifold $(X, \eta)$ of $\hskn$-type with an involution $i \in \aut(X)$ such that $\eta \circ i^* = \rho \circ \eta$.
\end{theorem}
\begin{proof}
Let $S \subset L$ be the anti-invariant lattice of $\rho$, i.e.\ the orthogonal complement of $T$. By \cite[Prop.~5.3]{BCS} the very general point $\omega \in \Omega_S$ is the image under the period map of a $T$-polarized marked manifold of $\hskn$-type $(X, \eta)$ with $\NS(X) = \eta^{-1}(T)$. We can then consider $\alpha:=\eta^{-1}\circ \rho\circ \eta\in O(H^2(X, \IZ))$, which is an involution, and we observe that:
\begin{enumerate}
\item $\alpha$ induces a Hodge isometry on $H^2(X, \IC)$ since the period point $\eta(H^{2,0}(X))$ is invariant for the action of $\rho$ on $\Omega_S$;
\item $\alpha$ is effective, because the equality $\NS(X) = \eta^{-1}(T) = \eta^{-1}(L^\rho)$ implies that there is an $\alpha$-fixed K\"ahler (even ample) class on $X$;
\item $\pm \rho \in \widetilde{O}(L)$.
\end{enumerate}
Hence, $\alpha$ is a monodromy operator by \cite[Lemma~9.2]{markman} and, by \cite[Thm.~1.3]{markman}, there exists $i \in \aut(X)$ such that $i^* = \alpha$. Since the map $\aut(X) \longrightarrow O(H^2(X, \IZ))$, sending an automorphism to its action on $H^2(X, \IZ)$, is injective for manifolds of $\hskn$-type (see \cite[Prop.~10]{beauville_rmks} and \cite[Lemma~1.2]{mongardi_deform}), the automorphism $i$ is both unique and an involution. It is then straightforward to check that $\eta \circ \iota^* = \rho \circ \eta$.
\end{proof}

\section{Geography for IHS manifolds of small dimension}\label{sect: geography}

The aim of this section is to make some remarks on which families of large dimension one can expect from the results of the previous section. We first classify the admissible invariant lattices of rank one and two, and then we describe the geography of these cases for manifolds of $\hskn$-type when $n \leq 5$.

\subsection{Invariant sublattices of rank one and two}

Let $T, S$ be the invariant and co-invariant lattices of a non-symplectic involution of a manifold of $\hskn$-type. As we saw in Proposition \ref{prop: T o S 2-elementary}, either $S$ or $T$ is $2$-elementary, depending on the action of the involution on the discriminant group of $L$ (which is $\id$ or $-\id$ respectively). Assume that $S$ is $2$-elementary and consider it embedded in the Mukai lattice $M$ (the case where $T$ is $2$-elementary is similar). Starting from the signature of $S^{\perp_M}$, we can use \cite[Thm.~1.5.2]{dolgachev} to deduce the possible isometry classes for $S^{\perp_M}$. As observed in Lemma~\ref{lemma: orthogonal Lperp}, we have that $T$ is the orthogonal complement in $S^{\perp_M}$ of $L^{\perp}$: since we know this last explicitly (see \eqref{eq: discriminant lperp}), we can use \cite[Prop.~1.15.1]{nikulin} to classify all primitive embeddings $L^{\perp} \hookrightarrow S^{\perp_M}$ and to compute, in each case, the discriminant group of the orthogonal complement, i.e.\ $A_T$.

\subsubsection{Invariant sublattice of rank one}\label{subsec:rank one}

In this subsection we prove the following proposition, which describes the pairs $T$ and $S$ that can occur when $\rk(T) = 1$.

\begin{prop} \label{prop: rank one}
Let $X$ be a manifold of $K3^{[n]}$-type for some $n \geq 2$, and let $i \in \aut(X)$ be a non-symplectic involution. If the invariant lattice $T \subset H^2(X, \IZ)$ has rank one, then one of the following holds:
\begin{enumerate}
\item if $i^*$ acts as $\id$ on $A_{H^2(X, \IZ)}$, then $-1$ is a quadratic residue modulo $n-1$ and
\[T \cong \langle 2(n - 1) \rangle, \qquad S \cong U^{\oplus 2} \oplus E_8^{\oplus 2} \oplus \langle -2 \rangle \oplus \langle -2 \rangle;\]
\item if $i^*$ acts as $-\id$ on $A_{H^2(X, \IZ)}$, then $T \cong \langle 2 \rangle$ and
\begin{enumerate}
\item either $S \cong U^{\oplus 2} \oplus E_8^{\oplus 2} \oplus \langle -2(n - 1) \rangle \oplus \langle -2 \rangle$;
\item or $n \equiv 0 \pmod 4$ and 
\[S \cong U^{\oplus 2} \oplus E_8^{\oplus 2} \oplus \left( \begin{array}{cc}
-\frac{n}{2} & n - 1\\
n - 1 & -2(n - 1)
\end{array} \right).\]
\end{enumerate}
\end{enumerate}
\end{prop}

\begin{proof}
This result generalizes \cite[Prop.~5.1]{catt_autom_hilb}, which holds for non-natural involutions of Hilbert schemes of points on a generic projective $K3$ surface.

We deal first with the case where $T,S$ are the invariant and anti-invariant lattices of an involution whose action on the discriminant $A_L$ is the identity. This means that $S$ is $2$-elementary and that $T \oplus L^{\perp} \subset S^{\perp_M}$. Since both $T$ and $L^{\perp}$ have signature $(1, 0)$, we deduce that $S^{\perp_M}$ has signature $(2, 0)$. By \cite[Table 15.1]{conway_sloane}, there is only one possible choice for $S^{\perp_M}$, which embeds in $M$ in a unique way by \cite[Thm.~1.1.2]{nikulin}: this is enough to claim that there is only one possible choice for $S$, up to isometries, which explicitly is
\[S = U^{\oplus 2} \oplus E_8^{\oplus 2} \oplus \langle -2 \rangle \oplus \langle -2 \rangle, \qquad S^{\perp_M} = \langle 2 \rangle \oplus \langle 2 \rangle.\]
We then need to look at how $L^{\perp} \cong \langle 2(n - 1) \rangle$ embeds primitively in $S^{\perp_M}$. A pair $(x, y)$ gives the coordinates of a primitive vector in $S^{\perp_M} = \langle 2 \rangle \oplus \langle 2 \rangle$ of square $2(n - 1)$ if and only if $\gcd(x, y) = 1$ and $x^2 + y^2 = n - 1$. Moreover, the isometry group of $S^{\perp_M}$ acts on these coordinates either by permutation or by exchanging sign. The orthogonal complement of $L^{\perp}$ in $S^{\perp_M}$, which is $T$, is then a lattice isometric to $\langle 2(n - 1) \rangle$, generated by $(-y, x)$. Notice that there exist two coprime integers $x, y$ such that $x^2 + y^2 = n - 1$ if and only if $-1$ is a quadratic residue modulo $n-1$ (to see this, combine \cite[Prop.~5.1.1]{Classical_Intro_Num_Theory} and \cite[Thm.~3.20]{Intro_Theory_Numbers}). 

We now consider the case where the action of $i^*$ on $A_L$ is $-\id$. We have that $T$ is $2$-elementary of signature $(1, 0)$, hence $T \cong \langle 2 \rangle$. It follows that $T$ embeds in a unique way in the Mukai lattice, with orthogonal complement
\[T^{\perp_M} \cong U^{\oplus 3} \oplus E_8^{\oplus 2} \oplus \langle -2 \rangle.\]
We now want to describe the different embeddings of $L^{\perp} \cong \langle 2(n - 1) \rangle$ in $T^{\perp_M}$. Since $T^{\perp_M}$ is unique in its genus, by \cite[Prop.~1.15.1]{NikulinInv} we have only two possibilities: they correspond to the two possible choices of a subgroup of $A_{T^{\perp_M}} \cong \IZ / 2\IZ$. Choosing the trivial subgroup, we see that the orthogonal complement of $L^{\perp}$ in $T^{\perp_M}$, i.e.~$S$, has discriminant group
\[A_S = \frac{\IZ}{2(n - 1) \IZ} \left( -\frac{1}{2(n - 1)} \right) \oplus \frac{\IZ}{2 \IZ} \left( -\frac{1}{2} \right),\]
and signature $(2, 20)$. By \cite[Thm.~2.4]{BCS}, there exists only one lattice with these invariants, up to isometries, which is
\[S = U^{\oplus 2} \oplus E_8^{\oplus 2} \oplus \langle -2(n - 1) \rangle \oplus \langle -2 \rangle.\]

The last possibility corresponds to the choice of the whole $A_{T^{\perp_M}}$, but in this case we must have $n \equiv 0 \pmod{4}$. This leads us to
\[A_S = \frac{\IZ}{(n - 1) \IZ} \left( -\frac{n}{2(n - 1)} \right),\]
where $S$ has again signature $(2, 20)$. By the same argument as above, there exists only one isometry class of lattices in this genus. A representative, which can be computed by applying \cite[Prop.~3.6]{ghs}, is
\begin{equation*}
S = U^{\oplus 2} \oplus E_8^{\oplus 2} \oplus \left( \begin{array}{cc}
-\frac{n}{2} & n - 1\\
n - 1 & -2(n - 1)
\end{array} \right).
\end{equation*}
\end{proof}

\begin{rem}
The three cases of Proposition \ref{prop: rank one} can be distinguished also by looking at the generator $t \in H^2(X, \IZ)$ of the invariant lattice $T$. In fact, by \cite[Prop.~3.6]{ghs}, we have that:
\begin{itemize}
\item in case (1), $t$ has square $2(n-1)$ and divisibility $n-1$;
\item in case (2a), $t$ has square $2$ and divisibility $1$;
\item in case (2b), $t$ has square $2$ and divisibility $2$.
\end{itemize} 
We point out that, by the Global Torelli Theorem for IHS manifolds, the existence of a primitive ample $t \in \textrm{NS}(X)$ with one of these three combinations of square and divisibility is sufficient to prove the existence of a non-symplectic involution on $X$, whose invariant lattice is $T = \langle t \rangle$ (see \cite[Prop.~5.3]{catt_autom_hilb}).
\end{rem}

\subsubsection{Invariant sublattice of rank two}\label{subsec:ranktwo}

The aim of this subsection is to provide some results for $\rk(T) = 2$. In particular, we describe the discriminant groups of the invariant and co-invariant lattices in complete generality, but we address the problem of their realization and uniqueness only for $n \leq 5$.

Assume that $\rk(T) = 2$, so that the signature of $T$ is $(1, 1)$. We first consider the case where the induced action on $A_L$ is the identity, hence $S$ is a $2$-elementary lattice of signature $(2, 19)$ and $S^{\perp_M}$ is $2$-elementary of signature $(2, 1)$. It follows from \cite[Thm.~1.1.2]{nikulin} that $S^{\perp_M}$ has a unique embedding in the Mukai lattice, up to isometries. By \cite[Thm.~1.5.2]{dolgachev} we have then two possibilities:
\[S^{\perp_M} = U \oplus \langle 2 \rangle \qquad \text{or} \qquad S^{\perp_M} = U(2) \oplus \langle 2 \rangle.\]

\begin{itemize}[leftmargin=0.35cm]
 \item {\bf Case $ \bm{S^{\perp_M} = U \oplus \langle 2 \rangle}$.} We look for a primitive embedding of $L^{\perp} = \langle 2(n - 1) \rangle$ in $S^{\perp_M}$. By \cite[Prop.~1.15.1]{nikulin} we need to consider pairs of isomorphic subgroups in $A_{L^{\perp}}$ and $A_{S^{\perp_M}} = \frac{\IZ}{2\IZ} \left( \frac{1}{2} \right)$. In particular, for the choice of the trivial subgroup we have
\[A_T = \frac{\IZ}{2(n - 1) \IZ} \left( -\frac{1}{2(n - 1)} \right) \oplus \frac{\IZ}{2 \IZ} \left( \frac{1}{2} \right).\]
A possible realization for this lattice $T$ is given by $T = \langle -2(n - 1) \rangle \oplus \langle 2 \rangle$; if $n \leq 5$, this is the only isometry class in the genus by \cite[Ch.~15, Thm.~21]{conway_sloane}.

The other possibility is to consider the subgroup of $A_{L^{\perp}}$ generated by the class of $n - 1$: in order for it to have the same discriminant form of $A_{S^{\perp_M}}$ we need $n \equiv 2 \pmod{4}$, and in this case we have
\[A_T = \frac{\IZ}{(n - 1) \IZ} \left( \frac{n - 2}{2(n - 1)} \right).\]
A lattice $T$ with this discriminant form and signature $(1,1)$ is the following:
\[T = \begin{pmatrix}
-2h & k\\
k & 2
\end{pmatrix}\]
\noindent where we write $n-1 = k^2 + 4h$, with $k, h$ non-negative integers and $k$ maximal. This is the only isometry class in the genus of $T$ if $n \leq 17$, by \cite[Ch.~15, Thm.~21]{conway_sloane}. For $n=2$, this lattice is isometric to $U$.

\item {\bf Case $ \bm{S^{\perp_M} = U(2) \oplus \langle 2 \rangle}$.} Here we have more possibilities, because there are more subgroups inside the discriminant group of $S^{\perp_M}$, which is
\[A_{S^{\perp_M}} = \left(\frac{\IZ}{2 \IZ} \right)^{\oplus 3}, \quad \text{with quadratic form} \; q_{S^{\perp_M}} = \begin{pmatrix}
0 & \frac{1}{2} & 0 \\
\frac{1}{2} & 0 & 0 \\
0 & 0 & \frac{1}{2}
\end{pmatrix}.\]
It is easy to see that we can discard the choice corresponding to the trivial subgroup, as it gives rise to a lattice $T$ of length $4$, hence the only relevant subgroups of $A_{S^{\perp_M}}$ are those of order two. Up to isomorphism, we have the two following possibilities.
\begin{enumerate}
 \item The subgroup is $\langle(0,0,1)\rangle\subset A_{S^{\perp_M}}$ with $q((0,0,1))=1/2$. This case can occur only if $n \equiv 2 \pmod{4}$, and gives
\[A_T = \frac{\IZ}{2(n - 1) \IZ} \oplus \frac{\IZ}{2 \IZ}, \quad \text{with quadratic form} \; q_{T} = \begin{pmatrix}
\frac{n-2}{2(n-1)} & \frac{1}{2} \\
\frac{1}{2} & 0
\end{pmatrix}.\]
For $n=2$, the lattice $U(2)$ realizes this genus; for $n=6$, we can consider the lattice whose bilinear form is given by the matrix $\begin{pmatrix}
2 & 4\\
4 & -2
\end{pmatrix}$. 
 \item The subgroup is $\langle v\rangle\cong\IZ/2\IZ\subset A_{S^{\perp_M}}$, for an element $v \neq (0,0,1)$ such that $q(v)=(n-1)/2$. This case gives 
 \[A_T = \frac{\IZ}{2(n - 1) \IZ} \left( -\frac{1}{2(n - 1)} \right) \oplus \frac{\IZ}{2 \IZ} \left( \frac{1}{2} \right).\]
A possible realization for this lattice is given by $T = \langle -2(n - 1) \rangle \oplus \langle 2 \rangle$; if $n \leq 5$, this is the only isometry class in the genus by \cite[Ch.~15, Thm.~21]{conway_sloane}.
\end{enumerate}
\end{itemize}

For $n \leq 5$, we summarize these results as follows.

\begin{prop}\label{prop:rank two +id}
Let $X$ be a manifold of $K3^{[n]}$-type for $2 \leq n \leq 5$, and let $i \in \aut(X)$ be a non-symplectic involution. If the invariant lattice $T \subset H^2(X, \IZ)$ has rank two and $i^*$ acts as $\id$ on $A_{H^2(X, \IZ)}$, then one of the following holds:
\begin{enumerate}
\item $T \cong \langle 2 \rangle \oplus \langle -2(n-1)\rangle$ and $S \cong U^{\oplus 2} \oplus E_8^{\oplus 2} \oplus \langle -2\rangle$;
\item $T \cong \langle 2 \rangle \oplus \langle -2(n-1)\rangle$ and $S \cong U \oplus U(2) \oplus E_8^{\oplus 2} \oplus \langle -2\rangle$;
\item $n = 2$, $T \cong U$ and $S \cong U^{\oplus 2} \oplus E_8^{\oplus 2} \oplus \langle -2\rangle$;
\item $n = 2$, $T \cong U(2)$ and $S \cong U \oplus U(2) \oplus E_8^{\oplus 2} \oplus \langle -2\rangle$.
\end{enumerate}
\end{prop}

We assume now that the action of the involution on the discriminant is $-\id$. In this case, $T$ is $2$-elementary of signature $(1, 1)$, so $T^{\perp_M}$ is also $2$-elementary and its signature is $(3, 19)$. This implies that $S$ (which is a sublattice of $T^{\perp_M}$) has signature $(2, 19)$. By \cite[Thm.~1.5.2]{dolgachev} there exist three $2$-elementary lattices of signature $(1, 1)$, namely $U$, $U(2)$ and $\langle 2 \rangle \oplus \langle -2 \rangle$. Every such lattice, by \cite[Thm.~1.1.2]{nikulin}, embeds in the Mukai lattice in a unique way, hence the orthogonal complement is uniquely determined too. We analyse the three cases separately: in each of them, there is only one isometry class in the genus of $S$ by \cite[Thm.~2.4]{BCS}.

\begin{itemize}[leftmargin=0.35cm]
\item {\bf Case $ \bm{T = U}$.} We have $T^{\perp_M} \cong U^{\oplus 3} \oplus E_8^{\oplus 2}$, which is unimodular. As a consequence, $L^{\perp} \cong \langle 2(n - 1) \rangle$ embeds in an essentially unique way in $T^{\perp_M}$ and its orthogonal complement $S$ is
\[S = U^{\oplus 2} \oplus E_8^{\oplus 2} \oplus \langle -2(n - 1) \rangle.\]

\item {\bf Case $ \bm{T = U(2)}$.} In this case, $T^{\perp_M} = U(2) \oplus U^{\oplus 2} \oplus E_8^{\oplus 2}$ has discriminant
\[A_{T^{\perp_M}} = \frac{\IZ}{2\IZ} \oplus \frac{\IZ}{2\IZ}, \quad \text{with quadratic form} \; q_{T^{\perp_M}} = \begin{pmatrix}
0 & \frac{1}{2} \\
\frac{1}{2} & 0
\end{pmatrix}.\]
As before, we look at the cyclic subgroups of $A_{T^{\perp_M}}$: a direct computation gives rise to two different cases.
\begin{enumerate}
\item If we choose the trivial subgroup we have $A_S = \frac{\IZ}{2(n - 1)\IZ} \oplus \frac{\IZ}{2\IZ} \oplus \frac{\IZ}{2\IZ}$, with quadratic form
\[q_S = \begin{pmatrix}
-\frac{1}{2(n - 1)} & 0 & 0\\
0 & 0 & \frac{1}{2}\\
0 & \frac{1}{2} & 0
\end{pmatrix}.\]
We conclude
\[S = U \oplus U(2) \oplus E_8^{\oplus 2}  \oplus \langle -2(n - 1) \rangle.\]
\item If $n \equiv 1,3 \pmod{4}$,  we can choose a subgroup  of order two and we have
\[A_S = \frac{\IZ}{2(n - 1) \IZ} \left( -\frac{1}{2(n - 1)} \right),\]
which corresponds to
\[S = U^{\oplus 2} \oplus E_8^{\oplus 2} \oplus \langle -2(n - 1) \rangle.\]

\end{enumerate}

\item {\bf Case $ \bm{T = \langle 2 \rangle \oplus \langle -2 \rangle}$.} Here $T^{\perp_M} = U^{\oplus 2} \oplus E_8^{\oplus 2} \oplus \langle 2 \rangle \oplus \langle -2 \rangle$, whose discriminant group is
\[A_{T^{\perp_M}} = \frac{\IZ}{2\IZ} \left( \frac{1}{2} \right) \oplus \frac{\IZ}{2\IZ} \left( -\frac{1}{2} \right).\]
The same kind of computations yield three cases:
\begin{enumerate}
\item The discriminant group is
\[A_S = \frac{\IZ}{2(n - 1) \IZ} \left( -\frac{1}{2(n - 1)} \right) \oplus \frac{\IZ}{2\IZ} \left( \frac{1}{2} \right) \oplus \frac{\IZ}{2\IZ} \left( -\frac{1}{2} \right),\]
which corresponds to
\[S = U \oplus E_8^{\oplus 2} \oplus \langle 2 \rangle \oplus \langle -2 \rangle \oplus \langle -2(n - 1) \rangle.\]
\item If $n \equiv 0,2 \pmod{4}$ we can have
\[A_S = \frac{\IZ}{2(n - 1) \IZ} \left( -\frac{1}{2(n - 1)} \right),\]
which is realized by
\[S = U^{\oplus 2} \oplus E_8^{\oplus 2} \oplus \langle -2(n - 1) \rangle.\]
\item If $n \equiv 1 \pmod{4}$  we can have
\[A_S = \frac{\IZ}{2(n - 1) \IZ} \left( \frac{n-2}{2(n - 1)} \right).\]
For $n=5$, a representative of the unique isometry class in this genus is
\[S = U \oplus E_8^{2} \oplus \begin{pmatrix}
-2 & 1 & 0\\
1 & -2 & 1 \\
0 & 1 & 2
\end{pmatrix}.\]
\end{enumerate}
\end{itemize}

The next proposition summarizes all possible pairs of lattices $T,S$ corresponding to involutions whose action on the discriminant group $A_L$ is $-\id$, for $n \leq 5$.

\begin{prop}\label{prop:rk two -id}
Let $X$ be  a manifold of $K3^{[n]}$-type  for $2 \leq n \leq 5$, and let $i \in \aut(X)$ be a non-symplectic involution. If the invariant lattice $T \subset H^2(X, \IZ)$ has rank two and $i^*$ acts as $-\id$ on $A_{H^2(X, \IZ)}$, then one of the following holds:
\begin{enumerate}
\item $T \cong U$ and $S \cong U^{\oplus 2} \oplus E_8^{\oplus 2} \oplus \langle -2(n-1) \rangle$;
\item $T \cong U(2)$ and $S \cong U \oplus U(2) \oplus E_8^{\oplus 2} \oplus \langle -2(n-1) \rangle$;
\item $T \cong \langle 2 \rangle \oplus \langle -2\rangle$ and $S \cong U \oplus E_8^{\oplus 2} \oplus \langle 2 \rangle\oplus \langle -2 \rangle \oplus \langle -2(n-1) \rangle$;
\item $n \in \left\{3, 5 \right\}$, $T \cong U(2)$ and $S \cong U^{\oplus 2} \oplus E_8^{\oplus 2} \oplus \langle -2(n-1) \rangle$;
\item $n \in \left\{2, 4 \right\}$, $T \cong \langle 2 \rangle \oplus \langle -2\rangle$ and $S \cong U^{\oplus 2} \oplus E_8^{\oplus 2} \oplus \langle -2(n-1) \rangle$;
\item $n = 5$, $T \cong \langle 2 \rangle \oplus \langle -2\rangle$ and $S \cong U \oplus E_8^{2} \oplus \begin{pmatrix}
-2 & 1 & 0\\
1 & -2 & 1 \\
0 & 1 & 2
\end{pmatrix}.$
\end{enumerate}
\end{prop}

\begin{rem}
For $n=2$, the isometries $\id$ and $-\id$ of $A_L \cong \IZ/2 \IZ$ coincide, hence Proposition \ref{prop:rank two +id} and Proposition \ref{prop:rk two -id} give the same classification (to check this, recall that $U(2) \oplus \langle -2 \rangle \cong \langle 2 \rangle \oplus \langle -2\rangle \oplus \langle -2\rangle$ by \cite[Thm.~1.5.2]{dolgachev}).
\end{rem}

\subsection{Deformation types for families of large dimension}
The lattice computations of Section \ref{subsec:rank one} and Section \ref{subsec:ranktwo} allow us to determine all moduli spaces $\mathcal{M}_{T, \rho}$, for $T$ an admissible invariant sublattice of rank one or two inside $L$ (recall the definitions from Section \ref{sec: existence}). By construction, the moduli spaces $\mathcal{M}_{T, \rho}$ arise as subspaces of the complex space $\mathcal{M}_{L}$, which parametrizes marked IHS manifolds of $\hskn$-type. The following fact was remarked in \cite[Theorem 9.5]{AST} for $K3$ surfaces, and it can be easily generalized to manifolds of $\hskn$-type.

\begin{lemma}\label{lem:closure}
Let $T', T'' \subset L$ be the invariant lattices of two monodromy operators $\rho', \rho'' \in \mon^2(L)$, respectively, and let $S' = (T')^\perp, S'' = (T'')^\perp$ be their orthogonal complements in $L$. The moduli space $\mathcal{M}_{T',\rho'}$ is in the closure of $\mathcal{M}_{T'',\rho''}$ if and only if $S' \subset S''\subset L$ and $(\rho'')\vert_{S'}=(\rho')\vert_{S'}$.
\end{lemma}

\begin{rem}
In our setting we can slightly improve the result of Lemma \ref{lem:closure}. In fact, as observed in Section \ref{sec: existence}, the orthogonal sublattices $T, S \subset L$ determine the involution $\rho \in \mon^2(L)$ as the unique extension of $\id_{T} \oplus (-\id_{S})$ to $L$. So, if we assume that $S'\subset S''$, then
\[(\rho'')\vert_{S'} = (-\id_{S''})\vert_{S'} = -\id_{S'} = (\rho')\vert_{S'}.\]
In the case of involutions we can then say that $\mathcal{M}_{T',\rho'}$ is in the closure of $\mathcal{M}_{T'',\rho''}$ if and only if $S' \subset S''\subset L$, as embedded sublattices.
\end{rem}

In this sense, the moduli spaces $\mathcal{M}_{T,\rho}$ of maximal dimension (where maximality is with respect to this notion) correspond to minimal (with respect to inclusion) admissible sublattices $T \subset L$. This is the reason why, in the previous section, we investigated in detail admissible invariant lattices of low rank. Any of these admissible lattices $T$ will give rise to at least one (but there could be more a priori, depending on the number of connected components of the moduli space) projective family of dimension $21 - \rk(T)$, whose generic member has a non-symplectic involution with invariant lattice $T$. We are now interested in computing the number of irreducible components for some of these moduli spaces.

We adopt the notation of \cite[Chapter 4]{joumaah}. Let $T \subset L$ be an admissible sublattice, i.e.\ the (hyperbolic) invariant lattice of an involution $\rho \in \mon^2(L)$, and let $\mathcal{C}_T$ be one of the two connected components of the cone $\left\{x \in T \otimes \IR \mid (x,x) > 0 \right\}$. The \emph{K\"ahler-type chambers of $T$} are the connected components of
\[ \mathcal{C}_T \setminus \bigcup_{\delta \in \Delta(T)} \delta^\perp \]
\noindent where $\Delta(T)$ is the set of wall divisors in $T$. As before, let $\Gamma_T$ be the image of the restriction map $\mon^2(L,T) \ra O(T)$: the subgroup $\Gamma_T \subset O(T)$ has finite index and it conjugates invariant wall-divisors, therefore it also acts on the set $\textrm{KT}(T)$ of K\"ahler-type chambers of $T$ (see \cite[\S 4.7]{joumaah}). In \cite[Theorem 4.8.11]{joumaah}, Joumaah proved that the quotient $\textrm{KT}(T)/\Gamma_T$ is in one-to-one correspondence with the set of distinct deformation types of marked manifolds $(X, \eta) \in \mathcal{M}_{T, \rho}$.

\begin{prop}\label{prop:U(2)irreducible}
Let $T \cong U(2)$ be a primitive sublattice of $L = U^{\oplus 3} \oplus E_8^{\oplus 2} \oplus$ $\langle -2(n-1)\rangle$ with orthogonal complement $S \cong U \oplus U(2) \oplus E_8^{\oplus 2} \oplus \langle -2(n-1) \rangle$. Let $\rho_1 \in \mon^2(L)$ be the involution which extends $\id_T \oplus (-\id_S)$. Then, for any $n \geq 2$ there is a single deformation type of marked manifolds of $\hskn$-type $(X, \eta) \in \mathcal{M}_{T, \rho_1}$.
\end{prop}

\begin{proof}
As we recalled above, the number of deformation types of $(\rho_1, T)$-polarized marked manifolds of $\hskn$-type is equal to the number of orbits of K\"ahler-type chambers of $T$, with respect to the action of the subgroup $\Gamma_T \subset O(T)$. For $T \cong U(2)$ as in the statement, an element $\delta \in T$ of coordinates $(a,b)$ with respect to a basis has square $4ab$ and divisibility in $L$ equal to $\gcd(a,b)$. In particular, the divisibility can only be one, if $\delta$ is primitive.  However, a direct computation using \cite[Thm.~12.1]{bayer_macri_mmp} shows that, if $\delta$ is a wall-divisor with $\divi(\delta) = 1$, then $\delta^2 = -2$ (see \cite[Rmk.~2.5]{mongardi_cones}). We conclude that there are no wall-divisors $\delta \in T$, since $T \cong U(2)$ contains no elements of square $-2$.
\end{proof}

As we showed in Subsection \ref{subsec:ranktwo}, when $n$ is odd there is a second way to embed the lattice $U(2)$ in $L$, which is not isometric to the one studied in Proposition \ref{prop:U(2)irreducible}. 

\begin{prop}\label{prop:U(2) 3 chambers}
For $n$ odd, let $T \cong U(2)$ be a primitive sublattice of $L = U^{\oplus 3} \oplus E_8^{\oplus 2} \oplus \langle -2(n-1)\rangle$ with orthogonal complement $S \cong U^{\oplus 2} \oplus E_8^{\oplus 2} \oplus \langle -2(n-1) \rangle$. Let $\rho_2 \in \mon^2(L)$ be the involution which extends $\id_T \oplus (-\id_S)$. Then, if $n=5$ there are three distinct deformation types of marked manifolds $(X, \eta) \in \mathcal{M}_{T, \rho_2}$.
\end{prop}

\begin{proof}
As in the proof of Proposition \ref{prop:U(2)irreducible}, we need to study the K\"ahler-type chambers of $T$ and therefore determine whether the lattice contains any wall-divisors. Up to isometries, the embedding $U(2) \hookrightarrow L$ in the statement can be realized as follows. Let $t = \frac{n-1}{2} \in \IN$ and consider the map
\[ j: U(2) \hookrightarrow L = U^{\oplus 3} \oplus E_8^{\oplus 2} \oplus \langle -2(n-1) \rangle, \quad (a,b) \mapsto 2ae_1 + (at+b)e_2 + ag\]
\noindent where $\left\{ e_1, e_2 \right\}$ is a basis for one of the summands $U$ of $L$ and $g$ is a generator of $\langle -2(n-1) \rangle$. We then have $j(U(2))^\perp \cong U^{\oplus 2} \oplus E_8^{\oplus 2} \oplus \langle -2(n-1) \rangle$, as requested. In particular, if $n=5$ (i.e.\ $t=2$) one can show that the divisibility in $L$ of $(a,b) \in T = j(U(2))$ is $\gcd(2a,b)$, hence, if the element is primitive, it can only be one or two. We compute explicitly all possible pairs $(\delta^2, \divi(\delta))$ for wall-divisors $\delta \in L_5 = U^{\oplus 3} \oplus E_8^{\oplus 2} \oplus \langle -8 \rangle$. This is an application of of \cite[Thm.~12.1]{bayer_macri_mmp} and \cite[Thm.~1.3]{mongardi_cones}, which gives the following results:
\begin{center}
\begin{tabular}{|c|c|}
\hline
$\delta^2$ & $\divi(\delta)$\\\hline \hline
$-2$ & $1$ \\\hline 
$-8$ & $2$ \\\hline
$-8$ & $4$ \\\hline
$-8$ & $8$ \\\hline
$-16$ & $2$ \\\hline
$-40$ & $4$ \\\hline
$-72$ & $8$ \\\hline
$-136$ & $8$ \\\hline
$-200$ & $8$ \\\hline
\end{tabular}
\end{center}

Since for any $\delta \in T$ we have $\delta^2 \in 4\IZ$, the only pairs $(\delta^2, \divi(\delta))$ for wall-divisors $\delta \in T$  are $(\delta^2, \divi(\delta)) = (-8,2), (-16,2)$. Each of the two admissible pairs $(\delta^2, \divi(\delta))$ yields a single wall-divisor $\delta \in T$, whose orthogonal complement $\delta^\perp$ intersects the positive cone of $T$ in its interior. We therefore have two (distinct) walls, which cut out three K\"ahler-type chambers in $\mathcal{C}_T$. These three chambers correspond to three distinct orbits, with respect to the action of the group $\Gamma_T$ on $\text{KT}(T)$. This is due to the fact that an isometry $\gamma \in \Gamma_T$ permutes the walls of the chambers, which in our case are generated by primitive vectors having all different squares.
\end{proof}

\begin{rem} \label{rem: moduli spaces recap}
By Proposition \ref{prop: rank one}, there are two distinct $(\rho, T)$-polarizations with $T \cong \langle 2 \rangle$. In the following, we will denote them by $(\rho_a, \langle 2 \rangle)$ and $(\rho_b, \langle 2 \rangle)$, where the orthogonal complement $S$ of the admissible sublattice $T \subset L$ is as in case ($2a$) and ($2b$), respectively, of the proposition. In particular, for all $n \geq 2$ the moduli space $\mathcal{M}_{\langle 2\rangle, \rho_a}$ is non-empty, while $\mathcal{M}_{\langle 2\rangle, \rho_b} = \emptyset$ if $n \not \equiv 0 \pmod{4}$. In turn, again by Proposition \ref{prop: rank one}, for $n \geq 3$ there is only one $(\rho, T)$-polarization with $T \cong \langle 2(n-1) \rangle$: we denote by $\mathcal{M}_{\langle 2(n-1)\rangle, \rho}$ the corresponding moduli space, which is non-empty if and only if $-1$ is a quadratic residue modulo $n-1$. Finally, for $T \cong U(2)$, we have the two polarizations $(\rho_1, U(2))$, $(\rho_2, U(2))$ which we studied in Proposition \ref{prop:U(2)irreducible} and Proposition \ref{prop:U(2) 3 chambers}, respectively.
\end{rem}

\begin{theorem} \label{thm: max dim families}
Let $(X, \eta)$ be a marked manifold of $K3^{[n]}$-type for $2 \leq n \leq 5$, and let $i \in \aut(X)$ be a non-symplectic involution such that the pair $(X, i)$ deforms in a family of dimension $d \geq 19$. Then $(X, \eta)$ belongs to the closure of one of the following moduli spaces.
 
 \begin{enumerate}
  \item[$n=2$:] $\mathcal{M}_{\langle 2\rangle, \rho_a}$ or $\mathcal{M}_{U(2), \rho_1}$.
  \item[$n=3$:] $\mathcal{M}_{\langle 2\rangle, \rho_a}$, $\mathcal{M}_{\langle 4\rangle, \rho}$ or $\mathcal{M}_{U(2), \rho_1}$.
  \item[$n=4$:] $\mathcal{M}_{\langle 2\rangle,\rho_a}$, $\mathcal{M}_{\langle 2\rangle,\rho_b}$  or $\mathcal{M}_{U(2), \rho_1}$.
  \item[$n=5$:] $\mathcal{M}_{\langle 2\rangle, \rho_a}$, $\mathcal{M}_{U(2), \rho_1}$ or $\mathcal{M}_{U(2), \rho_2}$.
 \end{enumerate} 
 
All these moduli spaces are irreducible with the exception of $\mathcal{M}_{U(2), \rho_2}$ for $n=5$, which has three distinct irreducible components.
\end{theorem}
\begin{proof}
Since $(X,i)$ deforms in a family of dimension at least $19$, it is a pair of type $T$ for some admissible lattice $T$ with $\rk(T) \leq 2$. At the level of period domains, the list in the statement is an easy consequence of Lemma \ref{lem:closure} and of Propositions \ref{prop: rank one}, \ref{prop:rank two +id} and \ref{prop:rk two -id}. Moreover, the period map is generically injective when restricted to manifolds polarized with a lattice of rank one, and the same is true in the case of $U(2)$ by Proposition \ref{prop:U(2)irreducible} and by \cite[Corollary 4.9.6]{joumaah}, with the exception of $n=5$ and $\mathcal{M}_{U(2), \rho_2}$ as explained in Proposition \ref{prop:U(2) 3 chambers}.
\end{proof}

\section{Examples}\label{sect: examples}
Even when we limit ourselves to $n \leq 5$, we observe that we lack the description of most of the projective families listed in Theorem \ref{thm: max dim families}. Indeed, while for $n=2$ both families have been described, respectively in \cite{OG}-\cite{BCMS} and \cite{IKKR_U2}, for $n\geq 3$ the family of $(\langle 2\rangle, \rho_a)$-polarized manifolds of $K3^{[n]}$-type is still unknown. In fact, when $n\geq 3$ the only two explicit examples which have been found are for $n=3$, $T\cong \langle 4\rangle$ (see \cite{IKKR} and Section \ref{sec: discr groups}) and $n=4$, $T\cong \langle 2\rangle$ with polarization $\rho_b$ (involution of the Lehn--Lehn--Sorger--van Straten eightfold; see for instance \cite{llms_twisted_cubics}), in addition to the involutions of Hilbert schemes of points on generic projective $K3$ surfaces whose existence has been proved by the second author in \cite{catt_autom_hilb}.

We conclude by observing that all families of dimension $19$ can in fact be realized as families of moduli spaces of stable twisted sheaves on a $K3$ surface. We briefly recall the construction and the properties of these moduli spaces.

Let $\Sigma$ be a $K3$ surface. By \cite[\S 2]{vg_brauer}, a Brauer class $\alpha \in H^2(\Sigma, \mathcal{O}^*_\Sigma)_{\text{tor}}$ of order $2$ corresponds to a surjective homomorphism $\alpha: \transc(\Sigma) \ra \IZ/2 \IZ$, where $\transc(\Sigma) = \ns(\Sigma)^\perp \subset H^2(\Sigma, \IZ)$ is the transcendental lattice of the surface. A $B$-field lift of $\alpha$ is a class $B \in H^2(\Sigma, \IQ)$ (which can be determined via the exponential sequence) such that $2B \in H^2(\Sigma, \IZ)$ and $\alpha(v) = (2B, v)$ for all $v \in \transc(\Sigma)$ (see \cite[\S 3]{stellari_huybrechts}). Notice that $B$ is defined only up to an element in $H^2(\Sigma, \IZ) + \frac{1}{2}\ns(\Sigma)$.

The full cohomology $H^*(\Sigma, \IZ) = H^0(\Sigma, \IZ) \oplus H^2(\Sigma, \IZ) \oplus H^4(\Sigma, \IZ)$, endowed with the pairing $(r,H,s)\cdot (r',H',s') = H\cdot H' - rs' -r's$, is a lattice isometric to the Mukai lattice $M = U^{\oplus 4} \oplus E_8^{\oplus 2}$. A Mukai vector $v = (r,H,s)$ is said to be \emph{positive} if $H \in \Pic(\Sigma)$ and either $r > 0$, or $r = 0$ and $H \neq 0$ effective, or $r = H = 0$ and $s > 0$. If $v=(r,H,s) \in H^*(\Sigma, \IZ)$ is positive, and $B$ is a $B$-field lift of $\alpha$, we define the twisted Mukai vector \mbox{$v_B \coloneqq (r, H + rB, s + B \cdot H + r\frac{B^2}{2})$}. If $v_B$ is primitive, for a suitable choice of a polarization $D$ of $\Sigma$ the coarse moduli space $M_{v_B}(\Sigma, \alpha)$ of $\alpha$-twisted Gieseker $D$-stable sheaves with Mukai vector $v_B$ is a projective IHS manifold of $\hskn$-type, with $n = \frac{v_B^2}{2} + 1$. Moreover, the image of the canonical embedding $H^2(M_{v_B}(\Sigma, \alpha), \IZ) \hookrightarrow M$, which we recalled at the beginning of Section \ref{sec: inv and anti-inv lattices}, is the subspace $v_{B}^\perp \subset M$  (see \cite{yoshioka_twisted} and \cite{baymacr}). For the sake of readability, we do not specify the ample divisor $D$ in the notation for $M_{v_B}(\Sigma, \alpha)$, even though the construction depends on it: we will always assume that a choice of a polarization (generic with respect to the Mukai vector $v_B$, in the sense of \cite[Def.~3.5]{yoshioka_twisted}) has been made. The transcendental lattice of $M_{v_B}(\Sigma, \alpha)$ is isomorphic to $\ker(\alpha) \subset \transc(\Sigma)$, which is a sublattice of index $2$ if $\alpha$ is not trivial. In turn, $\Pic(M_{v_B}(\Sigma, \alpha)) \cong v_B^\perp \cap \Pic(\Sigma, \alpha)$ inside $H^*(\Sigma, \IZ)$, where $\Pic(\Sigma, \alpha)$ is the sublattice generated by $\Pic(\Sigma)$ and by the vectors $(0,0,1), (2, 2B, 0)$ (see \cite[\S 3]{yoshioka_twisted} and \cite[Lemma 3.1]{macri_stellari}). 

\begin{prop}\label{prop: twisted induced U(2) irred}
For $n \geq 2$, let $(X, \eta)$ be a very general element in the moduli space $\mathcal{M}_{U(2), \rho_1}$ of Proposition \ref{prop:U(2)irreducible}, such that $\eta(\Pic(X)) \cong U(2)$. Then, the manifold $X$ is isomorphic to a moduli space of twisted sheaves on a very general projective $\langle 2(n-1) \rangle$-polarized $K3$ surface. 
\end{prop}

\begin{proof}
Let $\Sigma$ be a generic projective $K3$ surface of degree $2(n-1)$, i.e.\ $\Pic(\Sigma) = \IZ L$ with $L = \mathcal{O}_\Sigma(H)$ for an effective, ample divisor $H$ with $H^2 = 2(n-1)$. Let $\left\{e_1, e_2 \right\}$ generate one of the summands $U$ in $\transc(\Sigma) \cong U^{\oplus 2} \oplus E_8^{\oplus 2} \oplus \langle -2(n-1) \rangle$, and consider the Brauer class of order two:
\[ \alpha: \transc(\Sigma) \ra \IZ/ 2 \IZ, \qquad v \mapsto (e_1, v).\]
Clearly, $B = \frac{e_1}{2} \in H^2(\Sigma, \IQ)$ is a $B$-field lift of $\alpha$ such that $B^2 = 0$ and $B \cdot H = 0$, since $2B \in \transc(\Sigma)$. Consider the primitive positive Mukai vector $v = (0,H,0)$: then
\[v_B = \left( 0, H, B \cdot H \right) = v\]
\noindent and the moduli space $M_{v_B}(\Sigma, \alpha)$ is a manifold of $\hskn$-type with
\[\transc(M_{v_B}(\Sigma, \alpha)) \cong \ker(\alpha) \cong U \oplus U(2) \oplus E_8^{\oplus 2} \oplus \langle -2(n-1) \rangle.\]
Moreover, $\Pic(\Sigma, \alpha) = \langle (0,H,0), (0,0,1), (2, 2B, 0) \rangle \cong \langle 2 \rangle \oplus U(2)$, thus
\[ \Pic(M_{v_B}(\Sigma, \alpha)) \cong v_B^\perp \cap \Pic(\Sigma, \alpha) = \langle (0,0,1), (2, e_1, 0) \rangle \cong U(2). \]

Hence, the moduli space $Y = M_{v_B}(\Sigma, \alpha)$ constructed above has $\Pic(Y) \cong T$, $\transc(Y) \cong S$ for the lattices $T,S$ of Proposition \ref{prop:U(2)irreducible}. By the same proposition we know that the moduli space $\mathcal{M}_{U(2), \rho_1}$ is irreducible. For $(X, \eta) \in \mathcal{M}_{U(2), \rho_1}$ very general we also have $\Pic(X) \cong T$ and $\transc(X) \cong S$ (via the marking $\eta$). Hence, the statement follows from the generic injectivity of the period map for $U(2)$-polarized manifolds of $\hskn$-type (see \cite[Corollary 4.9.6]{joumaah}).
\end{proof}

\begin{rem}
For $(X, \eta) \in \mathcal{M}_{U(2), \rho_1}$, let $i \in \aut(X)$ be the non-symplectic involution such that $\eta \circ i^\ast= \rho_1 \circ \eta$. Even though, for $(X, \eta)$ very general, the manifold $X$ is isomorphic to $Y = M_{v_B}(\Sigma, \alpha)$ as in the previous proposition, if $n \geq 3$ we cannot realize the automorphism $i$ as a twisted induced involution  on $Y$ (in the sense of \cite{ckkm}), since the group of automorphisms of the $K3$ surface $\Sigma$ is trivial (see \cite[\S 5]{saint-donat}).
\end{rem}

\begin{prop} \label{prop: twisted induced U(2) n=5}
For $n = 5$, let $\mathcal{M}_{U(2), \rho_2}$ be the moduli space of Proposition \ref{prop:U(2) 3 chambers}. There exists an irreducible component $\mathcal{M}^0 \subset \mathcal{M}_{U(2), \rho_2}$ such that, for the very general element $(X,\eta) \in \mathcal{M}^0$ with $\eta(\Pic(X)) \cong U(2)$, the manifold $X$ is isomorphic to a moduli space $Y$ of twisted sheaves on a very general projective $\langle 2 \rangle$-polarized $K3$ surface. Moreover, the non-symplectic involution $i \in \aut(X)$ such that $\eta \circ i^\ast= \rho_2 \circ \eta$ is realized by a twisted induced automorphism on $Y$.
\end{prop}

\begin{proof}
Let $\Sigma$ be the double cover of $\IP^2$ branched along a smooth sextic curve. We have $\Pic(\Sigma) \cong \langle 2\rangle$ and $\transc(\Sigma) \cong U^{\oplus 2} \oplus E_8^{\oplus 2} \oplus \langle -2 \rangle$. If we denote by $g$ the generator of the summand $\langle -2 \rangle$ inside $\transc(\Sigma)$, then the (non-primitive) index two sublattice $U^{\oplus 2} E_8^{\oplus 2} \oplus \langle 2g \rangle \subset \transc(\Sigma)$ is isometric to $S = U^{\oplus 2} \oplus E_8^{\oplus} \oplus \langle -8\rangle$. Let $\alpha$ be the following Brauer class of order two:
\[ \alpha: \transc(\Sigma) \ra \IZ/ 2\IZ, \qquad \lambda + mg \mapsto m\]
\noindent where $\lambda \in U^{\oplus 2} \oplus E_8^{\oplus 2}$ and $m \in \IZ$. Clearly, $\ker(\alpha) = U^{\oplus 2} E_8^{\oplus 2} \oplus \langle 2g \rangle \cong S$. Let $\left\{e_1, e_2\right\}$ generate a summand $U$ inside $H^2(\Sigma, \IZ) \cong U^{\oplus 3} \oplus E_8^{\oplus 2}$. We can assume that $e_1 + e_2$ is the generator of $\Pic(\Sigma)$ and therefore $g = e_1 - e_2$. Notice that the rational class $B = \frac{e_2}{2} \in H^2(\Sigma, \IQ)$ is a $B$-field lift for $\alpha$, since $\alpha(x) = (e_2, x) \in \IZ/ 2\IZ$ for all $x \in \transc(\Sigma)$. Consider the (non-primitive) positive Mukai vector \mbox{$v = (0,2(e_1+e_2),0) \in H^*(\Sigma, \IZ)$}. When twisting $v$ with respect to the $B$-field lift $B$, we obtain $v_B = (0, 2(e_1+e_2), 1)$, which is now primitive of square $8$. Hence, the moduli space $M_{v_B}(\Sigma, \alpha)$ is a manifold of $K3^{[5]}$-type with transcendental lattice isomorphic to $S$. Moreover
\[\Pic(\Sigma, \alpha) = \langle (0,e_1+e_2,0), (0,0,1), (2, e_2, 0) \rangle\]
\noindent thus
\[ \Pic(M_{v_B}(\Sigma, \alpha)) \cong v_B^\perp \cap \Pic(\Sigma, \alpha) = \langle (0,0,1), (2, e_2, 0) \rangle \cong U(2).\]
Since $\Sigma$ is a double cover of the plane, it is equipped with a non-symplectic involution $\iota$, which acts as $\id$ on $H^0(\Sigma, \IZ) \oplus \Pic(\Sigma) \oplus H^4(\Sigma, \IZ)$ and as $-\id$ on $\transc(\Sigma)$. This implies that both the Brauer class $\alpha: \transc(\Sigma) \ra \IZ / 2\IZ$ and the twisted Mukai vector $v_B = (0, 2(e_1 + e_2), 1)$ are $\iota$-invariant. Then, by \cite[\S 3]{ckkm}, the moduli space $Y = M_{v_B}(\Sigma, \alpha)$ comes with a (non-symplectic) induced involution $\widetilde{\iota}$. In particular, the invariant lattice of $\widetilde{\iota}$ is the whole $\Pic(M_{v_B}(\Sigma, \alpha))$, since $\iota$ acts trivially on $\langle (0,0,1), (2,e_2,0) \rangle$ by \cite[Remark 2.4]{ckkm} (the two classes $(2,e_2,0)$ and $(2,\iota^*(e_2),0) = (2,e_1,0)$ coincide in $H^2(M_{v_B}(\Sigma, \alpha), \IZ)$). As in Proposition \ref{prop: twisted induced U(2) irred}, the statement follows from the generic injectivity of the period map, after recalling that $\mathcal{M}_{U(2), \rho_2}$ has three irreducible components by Proposition \ref{prop:U(2) 3 chambers}.
\end{proof}

\bibliographystyle{amsplain}
\bibliography{NonSympInvK3nBiblio}

\end{document}